\documentclass{article}

\usepackage{amsmath}
\usepackage[shortlabels]{enumitem}
\usepackage{amssymb}
\usepackage{amsthm}
\usepackage{tikz-cd}
\usepackage{mathrsfs}
\usepackage{authblk}

\newtheorem{thm}{Theorem}[section]
\newtheorem{cor}[thm]{Corollary}
\newtheorem{lem}[thm]{Lemma}
\newtheorem{prop}[thm]{Proposition}
\theoremstyle{remark}
\newtheorem{rmk}{Remark}
\theoremstyle{definition}

\newtheorem*{qst}{Question}

\DeclareMathOperator{\diff}{Diff}
\DeclareMathOperator{\aut}{Aut}
\DeclareMathOperator{\Hom}{Hom}
\DeclareMathOperator{\im}{im}

\DeclareMathOperator{\vect}{Vect}

\begin{document}

\title{On $1$-connected $8$-manifolds with the same homology as $S^3\times S^5$}
\author{Xueqi Wang}
\affil{Institute of Mathematics, Chinese Academy of Sciences, Bejing, 100190, China \authorcr Email: wangxueqi@amss.ac.cn}
\date{}
\maketitle

\begin{abstract}
In this article, we classify $1$-connected $8$-dimensional Poincar\'e complexes, topological manifolds and smooth manifolds with the same homology as $S^3\times S^5$. Some questions of Escher-Ziller are also discussed.
\end{abstract}

\section{Introduction}\label{sec:1}

The classification of manifolds is an interesting and challenging task in topology. However, so far, most of our knowledge is still concentrated on manifolds connected up to the middle dimension. In this article, we mainly concerns the following:

\begin{qst}
Give a classification of all $1$-connected $8$-manifolds with the same homology as $S^3\times S^5$.
\end{qst}

We'll call such $8$-manifolds of type ($\ast$). The motivation comes from a class of highly concerned $1$-connected $7$-manifolds. These manifolds satisfy either
\begin{enumerate}
  \item $H^2\cong \mathbb{Z}\{u\}, H^3=0, H^4\cong \mathbb{Z}_r\{u^2\}$, with $r\geq 1$, which will be called of type ($r$); or
  \item $H^2\cong \mathbb{Z}\{u\}, H^3=\mathbb{Z}\{v\}, H^4\cong \mathbb{Z}\{u^2\}, H^5\cong\mathbb{Z}\{uv\}$, i.e. the cohomology ring is isomorphic to that of $\mathbb{C}P^2\times S^3$, which will be called of type ($0$).
\end{enumerate}
Manifolds of type ($r$) are such as Aloff-Wallach manifolds, Eschenberg spaces, etc., which provide valuable examples of positive curved manifolds (see \cite{Zil14} for a survey). Complete invariants for such manifolds have been given in \cite{KS88}. The most natural examples of manifolds of type ($0$) may be the $S^3$-bundles over $\mathbb{C}P^2$ which have cross sections. The classification problem will be considered in our later paper \cite{W}. Now if we consider the $S^1$-bundle over a manifold of type ($r$) ($r\geq 0$) with Euler class $u$, then easy calculation shows that the total space is of type ($\ast$). Therefore, knowing more about manifolds of type ($\ast$) may give us more information on manifolds of type ($r$).

Let $\mathcal{P}$, $\mathcal{T}$, $\mathcal{S}$ be the homotopy equivalence, homeomorphism, diffeomorphism classes of Poincar\'e duality spaces, topological manifolds, smooth manifolds of type ($\ast$), respectively. Our main result is the following:
\begin{thm}\label{thm:16}
  \begin{enumerate}
    \item $\mathcal{P}=\{S^3\times S^5, X_{1,0}, X_{0,1}, X_{1,1}, SU(3)\}$. Here $X_{r,s}=(S^3\vee S^5)\cup_{[\iota_3,\iota_5]+ra_3\eta_6+s\eta_5\eta_6} D^8$, where $\iota_k$ corresponds to the identity map of $S^k$, $a_3\eta_6$ and $\eta_5\eta_6$ are certain elements in $\pi_7(S^3)$ and $\pi_7(S^5)$, respectively.
    \item $\mathcal{T}=\{S^3\times S^5, SU(3), M_{1,0}\}$, where $M_{1,0}$ is the total space of the only nontrivial $S^3$-bundle over $S^5$ which has a cross section. Note that $M_{1,0}\simeq X_{1,0}$.
    \item\label{1.3} $\mathcal{S}=\{S^3\times S^5, S^3\times S^5\#\Sigma^8, SU(3), M_{1,0}\}$, where $\Sigma^8$ is the only exotic $8$-sphere.
  \end{enumerate}
\end{thm}

As a consequence of Theorem \ref{thm:16}, we have the following rigidity theorem:

\begin{cor}
  \begin{enumerate}
    \item Two topological $8$-manifolds of type ($\ast$) are homeomorphic if and only if they are homotopy equivalent. All topological $8$-manifolds of type ($\ast$) are smoothable.
    \item A closed smooth $8$-manifold $M$ is diffeomorphic to $M_{1,0}$ or $SU(3)$ if and only if they are homotopy equivalent.
  \end{enumerate}
\end{cor}

As the $n$-th homotopy group of the total space of a circle bundle coincides with the base space, and obviously $\pi_i(M_{1,0})\cong \pi_i(S^3\times S^5)$, we have

\begin{cor}
  Let $N$ be a $7$-manifold of type ($r$). Then either $\pi_i(N)\cong \pi_i(S^3\times S^5), i\geq 3$ or $\pi_i(N)\cong \pi_i(SU(3)), i\geq 3$.
\end{cor}

The article is organized as follows. Section \ref{sec:2} gives some notations and elementary results which will be used later. Section \ref{sec:3} and \ref{sec:5} deals with the classification problem. Section \ref{sec:4} is a preparation for section \ref{sec:5}. Finally, in section \ref{sec:6}, we discuss some questions in a paper of Escher and Ziller \cite{EZ14}.

\section{Preliminaries}\label{sec:2}

We list here some notations and basic results which will be used later.

Let $x_0=(1,0,\ldots,0)^T$ be the base point of $S^n$. There is a principal fiber bundle
\[
SO(n)\xrightarrow{i}SO(n+1)\xrightarrow{p} S^n,
\]
defined by
\[
i:A\mapsto
\begin{pmatrix}
    1 &  \\
     & A
  \end{pmatrix},\quad
  p:B\mapsto Bx_0.
\]
Noticed that $SO(3)\xrightarrow{i} SO(4)\xrightarrow{p} S^3$ has a cross section $\sigma: S^3\to SO(4)$, given by the canonical identification of $S^3$ with $Sp(1)\subset SO(4)$. Explicitly,
\[
\sigma:
  \begin{pmatrix}
    a \\
    b \\
    c \\
    d
  \end{pmatrix}
  \mapsto
  \begin{pmatrix}
    a & -b & -c & -d \\
    b & a & -d & c \\
    c & d & a & -b \\
    d & -c & b & a
  \end{pmatrix}.
\]
This gives the diffeomorphism $SO(4)\cong SO(3)\times S^3$ and the isomorphism $\pi_4(SO(4))\cong\pi_4(SO(3))\oplus\pi_4(S^3)$.

Let $\iota_n\in\pi_n(S^n)$ be represented by the identity map of $S^n$, $\eta_2\in\pi_3(S^2)$ be represented by the Hopf map, $\eta_n=\Sigma^{n-2}\eta_2\in\pi_{n+1}(S^n)$.

\begin{lem}\label{thm:5}
  Notations as above.
  \begin{enumerate}
    \item $\pi_3(SO(3))\cong \mathbb{Z}\{b\}$ for some element $b$, \\
    $\pi_3(SO(4))\cong\mathbb{Z}\{i_*b\}\oplus\mathbb{Z}\{\sigma_*\iota_3\}$,\\
    $\pi_3(SO(5))\cong \mathbb{Z}\{a\}$, with $a=i_*\sigma_*\iota_3$;
    \item $\pi_4(SO(3))\cong \mathbb{Z}_2\{b\eta_3\}$, \\
    $\pi_4(SO(4))\cong\mathbb{Z}_2\{i_*b\eta_3\}\oplus \mathbb{Z}_2\{\sigma_*\eta_3\}$.
  \end{enumerate}
\end{lem}

Let $G_m=\{f:S^{m-1}\to S^{m-1}| \deg f=1\}$, $F_m=\{f\in G_m|f(x_0)=x_0\}$. We have the following ladder of fibration sequences:
\[
\begin{tikzcd}
SO(n) \arrow[r,"i"] \arrow[d] & SO(n+1) \arrow[d] \arrow[r,"p"] & S^n \arrow[d,equal]   \\
F_{n+1} \arrow[r,"\bar{i}"] & G_{n+1} \arrow[r,"\bar{p}"] & S^n
\end{tikzcd}
\]

We will denote the two maps $\pi_k(SO(n))\to\pi_k(F_{n+1})$ and $\pi_k(SO(n+1))\to \pi_k(G_{n+1})$ by $\bar{\mu}'$ and $\bar{\mu}$, respectively.

Define a map $I:\pi_n(F_m)\to \pi_{n+m}(S^{m})$ as follows. Write $S^{n+m}$ as $S^n\times D^m\cup D^{n+1}\times S^{m-1}$, and let $q:D^m\to D^m/\partial D^m\cong S^m$ be the quotient map. For any $\beta=[f]\in\pi_n(F_m)$, its image $I(\beta)$ is defined to be represented by
\[
I(f)(x,y)=
\begin{cases}
  f(x)(q(y)), & \mbox{if } (x,y)\in S^n\times D^m \\
  x_0, & \mbox{otherwise}.
\end{cases}
\]
By definition, there is a commutative diagram
\begin{equation}\label{eq:1}\tag{$\dag$}
  \begin{tikzcd}
  \pi_k(SO(n)) \arrow[rd,"J"] \arrow[d,"\bar{\mu}'"] & \\
  \pi_k(F_n) \arrow[r,"I"] & \pi_{k+n}(S^n)
  \end{tikzcd}
\end{equation}

\begin{thm}[\cite{Whi46}]\label{thm:6}
  $I$ is a group isomorphism.
\end{thm}

For homotopy groups of spheres and $J$-homomorphisms, we list the following results:

\begin{lem}[\cite{Tod52,Tod62}]\label{thm:18}
  \begin{enumerate}
    \item $\pi_{n+1}(S^n)\cong
    \begin{cases}
      \mathbb{Z}\{\eta_2\}, & \mbox{if } n=2 \\
      \mathbb{Z}_2\{\eta_n\}, & \mbox{if } n\geq 3,
    \end{cases}$
    \item $\pi_{n+2}(S^n)\cong\mathbb{Z}_2\{\eta_n\eta_{n+1}\}$ for $n\geq 2$;
    \item $\pi_6(S^3)\cong \mathbb{Z}_{12}\{a_3\}$, with $a_3=Jb$, \\
    $\pi_7(S^4)\cong \mathbb{Z}\{\nu_4\}\oplus\mathbb{Z}_{12}\{a_4\}$, with $\nu_4=J\sigma_*\iota_3$, $a_4=\Sigma a_3$;
    \item $\pi_7(S^3)\cong\mathbb{Z}_2\{a_3\eta_6\}$, $\eta_3\nu_4= a_3\eta_6$, $\eta_3a_4=0$,\\
    $\pi_8(S^4)\cong \mathbb{Z}_2\{\nu_4\eta_7\}\oplus\mathbb{Z}_2\{a_4\eta_7\}$, \\
    $\pi_{n+4}(S^n)=0$, $n\geq 6$.
  \end{enumerate}
\end{lem}

\begin{lem}\label{thm:17}
  \begin{enumerate}
    \item $\Sigma J=-Ji_*$;
    \item $J(\alpha\circ\beta)=J\alpha\circ\Sigma^m\beta$, where $\alpha\in \pi_k(SO(m)), \beta\in\pi_l(S^k)$.
  \end{enumerate}
\end{lem}

\begin{proof}
  \begin{enumerate}
    \item See \cite{Whi53}.
    \item Can be checked directly by definition.
  \end{enumerate}
\end{proof}

\begin{lem}\label{thm:8}
  \begin{enumerate}
    \item\label{8.1} $J:\pi_4(SO(3))\xrightarrow{\cong}\pi_7(S^3)$;
    \item $J:\pi_4(SO(4))\xrightarrow{\cong}\pi_8(S^4)$, with $J(\sigma_*\eta_3)=\nu_4\eta_7$, $J(i_*b\eta_3)=a_4\eta_7$.
  \end{enumerate}
\end{lem}

\begin{proof}
  Using Lemma \ref{thm:18} and \ref{thm:17}.
  \begin{enumerate}
    \item $J(b\eta_3)=(Jb)\eta_6=a_3\eta_6$.
    \item $J(\sigma_*\eta_3)=J(\sigma_*\iota_3)\eta_7=\nu_4\eta_7$.

  $J(i_*b\eta_3)=J(i_*b)\eta_7=(\sigma Jb)\eta_7=a_4\eta_7$.
  \end{enumerate}
\end{proof}

\begin{lem}\label{thm:10}
  $\bar{\mu}':\pi_4(SO(3))\xrightarrow{\cong}\pi_4(F_3)$, $\bar{\mu}:\pi_4(SO(4))\xrightarrow{\cong}\pi_4(G_4)$.
\end{lem}

\begin{proof}
  To see $\bar{\mu}'$ is an isomorphism, just combine Lemma \ref{thm:6}, \ref{thm:8} \ref{8.1} and (\ref{eq:1}).

  For $\bar{\mu}$, we have an exact ladder
  \[
  \begin{tikzcd}
  \pi_5(S^3) \arrow[r] \arrow[d,equal] & \pi_4(SO(3)) \arrow[r] \arrow[d,"\bar{\mu}'"] & \pi_4(SO(4)) \arrow[r] \arrow[d,"\bar{\mu}"] & \pi_4(S^3) \arrow[r,"\partial"] \arrow[d,equal] \arrow[l,bend left,"\sigma_*"] & \pi_3(SO(3)) \arrow[d,"\bar{\mu}'"] \\
  \pi_5(S^3) \arrow[r] & \pi_4(F_3) \arrow[r] & \pi_4(G_4) \arrow[r] & \pi_4(S^3) \arrow[r,"\bar{\partial}"] & \pi_3(F_3)
  \end{tikzcd}
  \]
  The top row is split, so $\partial=0$. Then $\bar{\partial}=\bar{\mu}'\partial=0$, and the 5-Lemma shows that $\bar{\mu}$ is an isomorphism.
\end{proof}

\section{The classification of Poincar\'e duality spaces and topological manifolds}\label{sec:3}

We prove the first two items in Theorem \ref{thm:16}. First is the classification of Poincar\'e duality spaces.

Let $X$ be a Poincar\'e duality space of type ($\ast$). Then
\[
X\simeq S^3\cup_{\phi_1} D^5\cup_{\phi_2} D^8.
\]
Since $\phi_1\in\pi_4(S^3)\cong\mathbb{Z}_2\{\eta_3\}$, it has two choices:
\begin{enumerate}
  \item $\phi_1=\eta_3$.

  In this case, $\phi_2\in \pi_7(S^3\cup_{\eta_3}S^5)=\pi_7(\Sigma\mathbb{C}P^2)\cong \mathbb{Z}\{\beta\}$ (see \cite{Muk82}). If $\phi_2=\pm\beta$, then $X\simeq SU(3)$. The cup product structure implies that it is the only possibility.
  \item $\phi_1=0$.

  In this case, $\phi_2\in  \pi_7(S^3\vee S^5)$. By Hilton's theorem \cite{Hil55}, $\pi_7(S^3\vee S^5)\cong \mathbb{Z}_2\{a_3\eta_6\} \oplus\mathbb{Z}_2\{\eta_5\eta_6\} \oplus\mathbb{Z}\{[\iota_3,\iota_5]\}$. Therefore, $\phi_2=ra_3\eta_6+s\eta_5\eta_6+t[\iota_3,\iota_5]$. The cup product structure is essentially determined by the term $t[\iota_3,\iota_5]$, which forces $t=\pm 1$, and we may always assume $t=1$. Let $X_{r,s}=(S^3\vee S^5)\cup_{ra_3\eta_6+s\eta_5\eta_6+[\iota_3,\iota_5]}D^8$, $r,s\in \{0,1\}$. To see that $X_{r,s}$ are different from each other, we use the following:
  \begin{lem}
  $X_{r,s}\simeq X_{r',s'}$ if and only if there exists $f\in \mathcal{E}(S^3\vee S^5)$ such that $f_*(ra_3\eta_6+s\eta_5\eta_6+[\iota_3,\iota_5])= \pm(r'a_3\eta_6+s'\eta_5\eta_6+[\iota_3,\iota_5])$, where $\mathcal{E}(X)$ denotes the group of self homotopy equivalences of $X$.
\end{lem}

\begin{proof}
  \lq\lq$\Rightarrow$" Let $g:X_{r,s}\to X_{r',s'}$ be a homotopy equivalence. We may assume $g$ is a cellular map. Let $f$ be the restriction of $g$ on the $5$-skeleton. Then $f\in \mathcal{E}(S^3\vee S^5)$, as $f$ induces isomorphisms on $H_*(S^3\vee S^5)$. Then the only if part follows by the diagram below:
  \[
  \begin{tikzcd}
  \mathbb{Z}\cong\pi_8(X_{r,s},S^3\vee S^5) \arrow[r, "\partial"] \arrow[d, "g_*","\cong"'] & \pi_7(S^3\vee S^5) \arrow[d,"f_*","\cong"'] \\
  \mathbb{Z}\cong\pi_8(X_{r',s'},S^3\vee S^5) \arrow[r,"\partial"] & \pi_7(S^3\vee S^5)
  \end{tikzcd}
  \]
  where $\partial$ sends $1$ to the attaching maps.

  \lq\lq$\Leftarrow$" The condition implies that $f$ can be extended to $g:X_{r,s}\to X_{r',s'}$. It is a homotopy equivalence, as it's easily checked that it induces an isomorphism between the cohomology rings.
\end{proof}

Clearly, $\mathcal{E}(S^3\vee S^5)\cong \{
\begin{pmatrix}
  \pm\iota_3 & \epsilon\eta_3\eta_4 \\
  0 & \pm\iota_5
\end{pmatrix}|\epsilon=0,1\}$.
We only verify the case $f=
\begin{pmatrix}
  \iota_3 & \eta_3\eta_4 \\
  0 & \iota_5
\end{pmatrix}$,
since others are with no difference.
\begin{gather*}
  f_*(a_3\eta_6)  =a_3\eta_6, \\
  f_*(\eta_5\eta_6)  =(\eta_3\eta_4+\iota_5)\eta_5\eta_6=\eta_5\eta_6, \\
  f_*([\iota_3,\iota_5])  =[\iota_3,\eta_3\eta_4+\iota_5]=[\iota_3,\eta_3\eta_4]+[\iota_3,\iota_5]=[\iota_3,\iota_5].
\end{gather*}
\end{enumerate}

Then we turn to the classification of topological manifolds. $M_{1,0}$ is the total space of the bundle with clutching map $i_*b\eta_3\in \pi_4(SO(4))$, therefore, together with \cite [(3.7)] {JW54} and Lemma \ref{thm:8}, we have
\[
M_{1,0}\simeq (S^3\vee S^5)\cup_{[\iota_3,\iota_5]+Jb\eta_3} D^8= (S^3\vee S^5)\cup_{[\iota_3,\iota_5]+a_3\eta_6} D^8=X_{1,0}.
\]
Then all we need to do is to show $X_{0,1}$ and $X_{1,1}$ are not homotopy equivalent to any topological manifolds.

The situation is similar to the well-known $5$-dimensional case (cf. \cite [32-33] {MM79}). We have
\[
\Sigma^n(X_{r,1})_+\simeq S^n\vee S^{n+3} \vee (S^{n+5}\cup_{\eta^2} D^{n+8})
\]
for $n$ sufficiently large. Let $\nu_{r,1}$ be the Spivak normal bundle of $X_{r,1}$, and $T_{r,1}$ be its Thom space. Then $T_{r,1}$ is the Spanier-Whitehead duality of $\Sigma^n(X_{r,1})_+$ \cite{Ati61}, hence
\[
T_{r,1}\simeq (S^l\cup_{\eta^2} D^{l+3}) \vee S^{l+5}\vee S^{l+8}.
\]
This implies that the thom space of $\nu_{r,1}|_{S^3}\simeq S^l\cup_{\eta^2} D^{l+3}$, hence can not be given a topological bundle structure, as $\pi_3(BTOP)=0$. Therefore, $\nu_{r,1}$ also does not have a topological bundle structure, which means $X_{r,1}$ is not homotopy equivalent to a topological manifold.

\section{Self equivalences of $S^k\times S^n$}\label{sec:4}

The key point of the classification of smooth manifolds is the observation: if $M^8$ is of type ($\ast$), then $M\cong S^3\times D^5\cup_f S^3\times D^5$, where $f\in \diff(S^3\times S^4)$. Then we'll analysis $\diff(S^3\times S^4)$, which relies heavily on results in \cite{Lev69}. We'll recall them here and add some easy observations. We always assume $k<n$ for simplicity, although the case $k=n$ was also considered in \cite{Lev69}.

We will deal with the following categories:
\begin{tabbing}
  $\mathscr{H}$ \= : \= Topological spaces and homotopy classes of maps, \\
  $\mathscr{D}$ \> : \> Smooth manifolds and smooth maps.
\end{tabbing}

Let $\bar{D}=\bar{D}^{k,n}$ ($\bar{H}=\bar{H}^{k,n}$) be the group of concordance (homotopy) classes of self-diffeomorphisms (self-homotopy equivalences) of $S^k\times S^n$. $\bar{A}$ may refer to any of them. By saying a self-equivalence, we mean a self-diffeomorphism or a self-homotopy equivalence, in its suited category. We have a natural homomorphism
$\mu: \bar{D}^{k,n}\to\bar{H}^{k,n}$, defined by considering a diffeomorphism merely as a homotopy equivalence.

Let $B=\aut H^*(S^k\times S^n)$ be the group of graded ring automorphisms of $H^*(S^k\times S^n)$. It is isomorphic to $\mathbb{Z}_2\oplus\mathbb{Z}_2$. Let $\Phi:\bar{A}^{k,n}\to B$ be the obvious homomorphism. Then $\Phi$ is onto, and its kernel $A=A^{k,n}$ is the subgroup of $\bar{A}^{k,n}$ which contains those orientation-preserving self-equivalences restricting to some $S^k\times x_0$ homotopic to the inclusion. Namely,
\[
1\to A\to \bar{A}\xrightarrow{\Phi} B\to 1
\]
Notice that $B$ can be realized by reflections on $S^k\times S^n$, which makes the short exact sequence split. So we have $\bar{A}\cong A\rtimes B$.

Define subgroups $A_1$, $A_2$ and $\alpha$ of $A$ to consist of those elements represented by $f:S^k\times S^n\to S^k\times S^n$ satisfying:
\begin{description}
  \item[$(A_1)$] $f$ extends to a self-equivalence of $D^{k+1}\times S^n$,
  \item[$(A_2)$] $f$ extends to a self-equivalence of $S^k\times D^{n+1}$,
  \item[$(\alpha)$] There is some $(k+n)$-disk $D\subset S^k\times S^n$ such that $f(D)\subset D$ and $f|_{S^k\times S^n-D}$ is the inclusion.
\end{description}
Subgroups $\bar{A}_1$, $\bar{A}_2$ and $\bar{\alpha}$ of $\bar{A}$ are similarly defined.

\begin{prop}[{\cite [Proposition 2.3] {Lev69}}]
  $\alpha=0$ in $\mathscr{H}$, while $\alpha\cong\Theta_{k+n+1}$, the group of exotic $(k+n+1)$-spheres, in $\mathscr{D}$.
\end{prop}

\begin{thm}[{\cite [Theorem 2.4] {Lev69}}]
  For $n\geq 3$, $A^{k,n}=(A_1\oplus\alpha)\rtimes_{\hat{\phi}} A_2$, with the action $\hat{\phi}$ trivial on $\alpha$. Besides, $A_1$, $A_2$, $\alpha$ are abelian.
\end{thm}

\begin{cor}
  For $n\geq 3$, $\bar{A}^{k,n}\cong (A_1\oplus\alpha)\rtimes \bar{A}_2$ and $\bar{A}_2\cong A_2\rtimes B$.
\end{cor}

\begin{proof}
  We already have
  \[
  \bar{A}\cong ((A_1\oplus\alpha)\rtimes A_2)\rtimes B.
  \]
  Actually, $A_1\oplus\alpha$ is a normal subgroup of $\bar{A}$ and $A_2 B$ is a subgroup of $\bar{A}$, which can be easily deduced from the observation that $A_1$, $A_2$ and $\alpha$ are invariant under the conjugation action of $B$. Therefore,
  \[
  \bar{A}\cong(A_1\oplus\alpha)\rtimes (A_2\rtimes B).
  \]

  The elements in $B$ are represented by reflections, which obviously extend to $S^k\times D^{n+1}$. Hence $A_2\rtimes B\subset\bar{A}_2$. They are actually equal, as $(A_1\oplus\alpha)\cap\bar{A}_2=(A_1\oplus\alpha)\cap A_2=\{1\}$.
\end{proof}

Now we turn to the determination of $A_1$, $A_2$ and the semi-direct product structure $\hat{\phi}$.

Let $FC_m^p$ be the group of concordance classes of framed imbeddings $S^m\hookrightarrow S^{m+p}$. We construct homomorphisms:
\[
\lambda_1: D_1^{k,n}\to FC_n^{k+1},\quad \lambda_2: D_2^{k,n}\to FC_k^{n+1}
\]
as follows. Let $f:S^k\times S^n\to S^k\times S^n$ represent an element $\xi$ in $D_1^{k,n}$, $F: D^{k+1}\times S^n\to D^{k+1}\times S^n$ be an extension of $f$. Then the framed embedding is defined by the composition of the following obvious maps:
\[
S^n\times D^{k+1} \xrightarrow[\cong]{F} S^n\times D^{k+1}\subset S^n\times\mathbb{R}^{k+1}\subset \mathbb{R}^{k+n+1},
\]
which represents the element $\lambda_1(\xi)\in FC_k^{n+1}$. $\lambda_2$ is defined similarly. Note that we have obvious homomorphisms:
\[
e_1:\pi_n(SO(k+1))\to D_1^{k,n},\quad e_2:\pi_k(SO(n+1))\to D_2^{k,n}
\]
defined by associating a smooth map $f:S^n\to SO(k+1)$ (or $g:S^k\to SO(n+1)$) to the diffeomorphism $(x,y)\mapsto(f(y)x,y)$ (or $(x,y)\mapsto(x,g(x)y)$), which fit into commutative diagrams:
\begin{equation}\label{com:1}
  \begin{tikzcd}[row sep=tiny]
\pi_n(SO(k+1)) \arrow[rd] \arrow[dd,"e_1"'] & \\
 & FC_n^{k+1} \\
D_1^{k,n} \arrow[ru,"\lambda_1"] &
\end{tikzcd}
\quad
\begin{tikzcd}[row sep=tiny]
\pi_k(SO(n+1)) \arrow[rd] \arrow[dd,"e_2"'] & \\
 & FC_k^{n+1} \\
D_2^{k,n} \arrow[ru,"\lambda_2"] &
\end{tikzcd}\tag{$\ast$}
\end{equation}
where $\pi_p(SO(m))\to FC_p^m$ is the same as in \cite [(5.10)] {Hae66}.

Define similar homomorphisms
\[
\lambda_1: H_1^{k,n}\to \pi_n(G_{k+1}),\quad \lambda_2: H_2^{k,n}\to \pi_k(G_{n+1})
\]
in $\mathscr{H}$ as follows. Any element of $H_1$ can be represented by a map of the form $(x,y)\mapsto(g(x,y),y)$, where $g: S^k\times S^n\to S^k$ corresponds to an element of $\pi_n(G_{k+1})$. This induces the homomorphism $\lambda_1$, and in a similar fashion, $\lambda_2$ is defined. It is clear that
\begin{equation}\label{com:2}
  \begin{tikzcd}
\pi_n(SO(k+1)) \arrow[r,"\bar{\mu}"] \arrow[d,"e_1"] & \pi_n(G_{k+1}) \\
D_1^{k,n} \arrow[r,"\mu"] & H_1^{k,n} \arrow[u,"\lambda_1"]
\end{tikzcd}
\quad
\begin{tikzcd}
\pi_k(SO(n+1)) \arrow[r,"\bar{\mu}"] \arrow[d,"e_2"] & \pi_k(G_{n+1}) \\
D_2^{k,n} \arrow[r,"\mu"] & H_2^{k,n} \arrow[u,"\lambda_2"]
\end{tikzcd}\tag{$**$}
\end{equation}

\begin{thm}[{\cite [Theorem 3.3] {Lev69}}]\label{thm:2}
  $\lambda_1,\lambda_2$ are isomorphisms, assuming $n\geq 3$ and $k\geq 2$.
\end{thm}

\begin{cor}\label{thm:9}
  $\pi_4(SO(4))\xrightarrow[\cong]{e_1} D_1^{3,4}\xrightarrow[\cong]{\mu} H_1^{3,4}$, $\pi_3(SO(5))\xrightarrow[\cong]{e_2} D_2^{3,4}$.
\end{cor}

\begin{proof}
  Recall the long exact sequence in \cite{Hae66}:
  \[
  \cdots\to \pi_n(SO(q))\to FC_n^q\to C_n^q\to \pi_{n-1}(SO(q))\to\cdots,
  \]
  where $C_n^q$ is the group of concordance classes of embeddings of $S^n$ in $S^{n+q}$.

  For the case $FC_3^5$, we have
  \[
  \cdots\to C_4^5\to \pi_3(SO(5))\to FC_3^5\to C_3^5\to \cdots.
  \]
  Using the fact that $C_n^q=0$ for $n<2q-3$ (see e.g. \cite[(6.6)]{Hae66}), $\pi_3(SO(5))\xrightarrow{\cong} FC_3^5$ follows.

  For $FC_4^4$, we have
  \[
  \cdots\to FC_5^4\to C_5^4\to \pi_4(SO(4))\to FC_4^4\to C_4^4\to \cdots.
  \]
  $C_4^4=0$ while $C_5^4\neq0$. But $FC_5^4\to C_5^4$ is surjective, as all $5$-spheres in $S^9$ have trivial normal bundles (\cite [Theorem 8.2] {Ker59}). Therefore, $\pi_4(SO(4))\xrightarrow{\cong} FC_4^4$.
  Then the corollary follows from Theorem \ref{thm:2}, Lemma \ref{thm:10} and commutative diagrams (\ref{com:1}) (\ref{com:2}).
\end{proof}

The action $\hat{\phi}$ can be divided into two parts, i.e.
\[
\hat{\phi}(g_2).g_1=\phi(g_2).g_1+\tau(g_2).g_1, \text{ for } g_i\in D_i,
\]
where $\phi:D_2\to \aut D_1$ is a homomorphism, and $\tau:D_2\to\Hom(D_1,\alpha)$ satisfies:
\[
\tau(gg')=\tau(g)\phi(g')+\tau(g'), \text{ for } g,g'\in D_2.
\]

In $\mathscr{H}$, the action $\phi$ can be completely determined, which will be presented below. Noticed that the homomorphism $\mu:\bar{D}\to\bar{H}$ preserves the action $\phi$, hence a large amount of information about $\phi$ in $\mathscr{D}$ can also be known.

Let $\theta=\bar{i}_*I^{-1}:\pi_{n+m}(S^{m})\xleftarrow[\cong]{I}\pi_n(F_m)\xrightarrow{\bar{i}_*} \pi_n(G_{m+1})$. We now identify $H_1^{k,n}$ with $\pi_n(G_{k+1})$ and $H_2^{k,n}$ with $\pi_k(G_{n+1})=\im\theta$, using Theorem \ref{thm:2} and the surjectivity of $\bar{i}_*:\pi_k(F_n)\to\pi_k(G_{n+1})$ when $k<n$.

\begin{prop}[{\cite [Proposition 4.2] {Lev69}}]\label{thm:11}
  $\phi(\theta(\xi)).\beta=\beta-\theta(\bar{p}_*(\beta)\circ\xi)$, for $\xi\in \pi_{n+k}(S^n)$ and $\beta\in\pi_n(G_{k+1})$.
\end{prop}

The function $\tau:D_2\to\Hom(D_1,\alpha)$, or in some different words, the pairing $\tau: D_1^{k,n}\otimes D_2^{k,n}\to\Theta_{k+n+1}$, can be viewed as a special case of the pairing
\[
T: \pi_n(SO(k+1))\otimes\pi_k(SO(n+1))\to\Theta_{k+n+1}
\]
studied by Milnor in \cite{Mil59}. Namely,
\[
\begin{tikzcd}
\pi_n(SO(k+1))\otimes\pi_k(SO(n+1)) \arrow[d,"e_1\otimes e_2"'] \arrow[dr,"T"] &  \\
D_1^{k,n}\otimes D_2^{k,n} \arrow[r,"\tau"] & \Theta_{k+n+1}
\end{tikzcd}
\]

\begin{lem}\label{thm:12}
  For $k=3, n=4$, $\tau$ coincide with $T$.
\end{lem}

\begin{proof}
  By Corollary \ref{thm:9}.
\end{proof}

\section{The classification of smooth manifolds} \label{sec:5}

\begin{lem}
  A smooth $8$-manifold $M$ is of type ($\ast$) if and only if $M\cong S^3\times D^5\cup_f S^3\times D^5$ for some $f\in \diff(S^3\times S^4)$.
\end{lem}

\begin{proof}
  The \lq\lq if \rq\rq\ part can be easily seen using van Kampen theorem and Mayer-Vietories sequence.

  For the \lq\lq only if\rq\rq\ part, first notice that there exists a self-indexed minimal Morse fuction $h: M\to\mathbb{R}$. Namely, $h$ has only $4$ critical points, with values $0,3,5,8$ respectively. Let $W_1=h^{-1}(-\infty, 4]$ and $W_2=h^{-1}[4, +\infty)$. Then $M=W_1\cup W_2$, and both $W_1$ and $W_2$ can be obtained by pasting a $3$-handle $D^3\times D^5$ on $D^8$. But orientation preserving framed imbeddings of $S^2$ in $S^7$ are all isotopic. Therefore $W_1\cong W_2\cong S^3\times D^5$.
\end{proof}

Let $M(f)=S^3\times D^5\cup_f S^3\times D^5$, where $f\in \diff(S^3\times S^4)$. Noticing that if $f_0, f_1 \in \diff(S^3\times S^4)$ are concordant, then $M(f_0)\cong M(f_1)$. Therefore, each $\alpha=[f]\in \bar{D}=\bar{D}^{3,4}$ defines an element $[M(f)]\in\mathcal{S}$, which is denoted by $M(\alpha)$.

\begin{prop}\label{thm:3}
  $M(\alpha_1)=M(\alpha_2)$ if and only if there exist $\beta_1, \beta_2\in\bar{D}_2$, such that $\beta_1\alpha_1=\alpha_2\beta_2$.

  In other words, $\mathcal{S}\cong \bar{D}/\sim$, where $\alpha_1\sim\alpha_2$ if there exist $\beta_1, \beta_2\in\bar{D}_2$, such that $\beta_1\alpha_1=\alpha_2\beta_2$.
\end{prop}

\begin{cor}\label{thm:4}
  Elements in $\mathcal{S}$ are 1-1 correspondence with the orbits of the conjugation action of $\bar{D}_2$ on $D_1\oplus\alpha$.
\end{cor}

Proposition \ref{thm:3} and Corollary \ref{thm:4} are acturally special cases of Lemma 5.4, Proposition 5.7 in \cite{Lev69}.

Now we analysis the action of $\bar{D}_2$. We'll only analysis the action of $D_2$, i.e. $\hat{\phi}$, as it will be seen that it's already enough for our final result.

\begin{prop}\label{thm:13}
  Identify $D_1, D_2$ with $\pi_4(SO(4)), \pi_3(SO(5))$, respectively. Notations are as in Lemma \ref{thm:5}. Then
  \begin{align*}
    \phi(a)(\sigma_*\eta_3) & =\sigma_*\eta_3+i_*b\eta_3, \\
    \phi(a)(i_*b\eta_3) & =i_*b\eta_3.
  \end{align*}
\end{prop}

\begin{proof}
  By corollary \ref{thm:9}, we can just passing to $\mathscr{H}$ to do calculations.
  \[
  \begin{split}
     \bar{\mu}(\phi(a)) & =\phi(\bar{\mu}(a))=\phi(\bar{\mu}(i_*\sigma_*\iota_3))=\phi(\bar{i}_*\bar{\mu}'\sigma_*\iota_3) \\
       & =\phi(\bar{i}_*I^{-1}J\sigma_*\iota_3)=\phi(\theta(J\sigma_*\iota_3))=\phi(\theta(\nu_4))
  \end{split}
  \]
  Then using the formula of Proposition \ref{thm:11},
  \[
  \begin{split}
     \bar{\mu}(\phi(a)(\sigma_*\eta_3)) & =\phi(\bar{\mu}(a))(\bar{\mu}(\sigma_*\eta_3)) \\
       & =\phi(\theta(\nu_4))(\bar{\mu}(\sigma_*\eta_3)) \\
       & =\bar{\mu}(\sigma_*\eta_3)-\theta(\bar{p}_*(\bar{\mu}(\sigma_*\eta_3))\circ \nu_4) \\
       & =\bar{\mu}(\sigma_*\eta_3)-\theta(p_*\sigma_*\eta_3\circ \nu_4) \\
       & =\bar{\mu}(\sigma_*\eta_3)-\theta(\eta_3\nu_4) \\
       & =\bar{\mu}(\sigma_*\eta_3)-\bar{i}_*I^{-1}(J(b\eta_3)) \\
       & =\bar{\mu}(\sigma_*\eta_3)-\bar{\mu}i_*(b\eta_3) \\
       & =\bar{\mu}(\sigma_*\eta_3-i_*b\eta_3) \\
       & =\bar{\mu}(\sigma_*\eta_3+i_*b\eta_3)
  \end{split}
  \]
  Similarly,
  \[
  \begin{split}
     \bar{\mu}(\phi(a)(i_*b\eta_3)) & =\phi(\theta(\nu_4))(\bar{\mu}(i_*b\eta_3)) \\
       & =\bar{\mu}(i_*b\eta_3)-\theta(\bar{p}_*(\bar{\mu}(i_*b\eta_3))\circ \nu_4) \\
       & =\bar{\mu}(i_*b\eta_3)-\theta(p_*(i_*b\eta_3)\circ \nu_4) \\
       & =\bar{\mu}(i_*b\eta_3)
  \end{split}
  \]

  Since $\bar{\mu}$ is an isomorphism, it completes the proof.
\end{proof}

For $\tau:D_2\to\Hom(D_1,\alpha)$, it is equivalent to analysis the Milnor pairing $T: \pi_4(SO(4))\otimes\pi_3(SO(5))\to\Theta_8$, by Lemma \ref{thm:12}. It was almost done in \cite{Fra68}.

\begin{prop}\label{thm:14}
  \[
    \begin{array}{cccc}
      T: & \pi_4(SO(4))\otimes\pi_3(SO(5)) & \longrightarrow & \Theta_8 \\
       & (0,a) & \mapsto & S^8 \\
       & (\sigma_*\eta_3,a) & \mapsto & S^8 \\
       & (i_*b\eta_3,a) & \mapsto & \Sigma^8 \\
       & (\sigma_*\eta_3+i_*b\eta_3,a) & \mapsto & \Sigma^8 \\
    \end{array}
  \]
\end{prop}

\begin{proof}
  Combine \cite [Lemma 1, Theorem 3]{Fra68} and Lemma \ref{thm:8} (2).
\end{proof}

\begin{cor}\label{thm:15}
  The orbits of $D_1\oplus\alpha$ under the action $\hat{\phi}$ are:
  \begin{gather*}
    \{0\},\{\Sigma^8\},\{i_*b\eta_3,i_*b\eta_3+\Sigma^8\}, \\
    \{\sigma_*\eta_3, \sigma_*\eta_3+i_*b\eta_3, \sigma_*\eta_3+\Sigma^8, \sigma_*\eta_3+i_*b\eta_3+\Sigma^8\}.
  \end{gather*}
\end{cor}

\begin{proof}[Proof of Theorem \ref{thm:16} \ref{1.3}]
  By Corollary \ref{thm:15}, there are at most $4$ elements in $\mathcal{S}$: $S^3\times S^5$, $S^3\times S^5\#\Sigma^8$, $SU(3)$ and $M_{1,0}$. As $\alpha\cong\Theta_8$ is invariant under the action of $B$, $0$ and $\Sigma^8$ will lie in different orbits at last. Thus $S^3\times S^5$ is not diffeomorphic to $S^3\times S^5\#\Sigma^8$. The theorem follows.
\end{proof}

Corollary \ref{thm:15} also implies:

\begin{cor}
  Let $S$ be the total space of the $S^3$-bundle over $S^5$ with clutching map $i_*b\eta_3+\sigma_*\eta_3$, then $S$ is diffeomorphic to $SU(3)$.
\end{cor}

\section{Further discussions}\label{sec:6}

As we've introduced in section \ref{sec:1}, for a $7$-manifold of type ($r$), the total space of the $S^1$-bundle with Euler class $u$ is an $8$-manifold of type ($\ast$). Conversely, if an $8$-manifold of type ($\ast$) admits a smooth free $S^1$-action, the orbit space will be a $7$-manifold of type ($r$).

In \cite{EZ14}, Escher and Ziller have given many remarks and conjectures concerning the total spaces of $S^1$-bundles over certain such $7$-manifolds. Unfortunately, some of them are wrong. We'll go through these examples in detail. It seems that finding a general way to determine the total spaces will be more helpful.

The first class of manifolds are $S^3$-bundles over $\mathbb{C}P^2$. It is well known that a $4$-dimensional vector bundle over $\mathbb{C}P^2$ is classified by its second Stiefel-Whitney class, first Pontryagin class and Euler class. Up to isomorphism, they can be constructed as follows. Recall that $\vect_{\mathbb{R}}^4(S^4)\cong \pi_3(SO(4))\cong \mathbb{Z}\{\alpha\}\oplus\mathbb{Z}\{\beta\}$ with
\begin{center}
  \begin{tabular}{c|c|c}
   & $\alpha$ & $\beta$ \\
  \hline
  $e$ & $0$ & $\omega_{S^4}$ \\
  \hline
  $p_1$ & $4\omega_{S^4}$ & $-2\omega_{S^4}$\\
\end{tabular}
\end{center}
where $\omega_{S^4}$ is the orientation cohomology class of $S^4$. Choose an embedded disk $D\cong D^4$ in $\mathbb{C}P^2$. Let $q: \mathbb{C}P^2\to S^4$ be the map collapsing the outer of $D$ to a point, and $v: \mathbb{C}P^2 \to \mathbb{C}P^2\vee S^4$ be the map just shrinking the boundary of $D$ to a point. Then all $4$-dimensional vector bundles over $\mathbb{C}P^2$ can be constructed by the following two pullback squares:
\[
\begin{matrix}
  \text{Spin case:} & \quad & \text{Nonspin case:} \\
  \begin{tikzcd}
q^*(\xi_{k,l}) \arrow[r] \arrow[d] & \xi_{k,l} \arrow[d] \\
\mathbb{C}P^2 \arrow[r,"q"] & S^4
\end{tikzcd} & \quad &
\begin{tikzcd}
v^*(\xi'_{k,l}) \arrow[r] \arrow[d] & \xi'_{k,l} \arrow[d] \\
\mathbb{C}P^2 \arrow[r,"v"] & \mathbb{C}P^2 \vee S^4
\end{tikzcd}
\end{matrix}
\]
where $\xi_{k,l}=k\alpha+l\beta$, $\xi'_{k,l}|_{\mathbb{C}P^2}=\gamma\oplus\epsilon^2$, $\xi'_{k,l}|_{S^4}=\xi_{k,l}$, with $\gamma$ the tautological line bundle over $\mathbb{C}P^2$. Now let $N_{k,l}=S(q^*(\xi_{k,l}))$, $N'_{k,l}=S(v^*(\xi'_{k,l}))$. Then $N_{k,l}$ and $N'_{k,l}$ are $1$-connected, and of type ($0$) if $l=0$, or of type ($r$) ($r\geq 1$) if $l\neq 0$. And we have $H^4(N_{k,l})\cong H^4(N'_{k,l})\cong \mathbb{Z}_l$.

Let $E_{k,l}$ and $E'_{k,l}$ be the total space of the $S^1$-bundle over $N_{k,l}$ and $N'_{k,l}$, respectively, with Euler class a chosen generator of $H^2$. Then we have the following pullback squares of fiber bundles:
\[
\begin{tikzcd}
 & S^3 \arrow[d] & S^3 \arrow[d] & S^3 \arrow[d] & S^3 \arrow[d] \\
S^1 \arrow[r] & E_{k,l} \arrow[r] \arrow[d] & N_{k,l} \arrow[d] \arrow[r] & S(\xi_{k,l}) \arrow[r] \arrow[d] & S(\gamma^4) \arrow[d] \\
S^1 \arrow[r] & S^5 \arrow[r,"\pi"] & \mathbb{C}P^2 \arrow[r,"q"] & S^4 \arrow[r,"f_{k,l}"] & BSO(4)
\end{tikzcd}
\]
\[
\begin{tikzcd}
 & S^3 \arrow[d] & S^3 \arrow[d] & S^3 \arrow[d] & S^3 \arrow[d] \\
S^1 \arrow[r] & E'_{k,l} \arrow[r] \arrow[d] & N'_{k,l} \arrow[d] \arrow[r] & S(\xi'_{k,l}) \arrow[r] \arrow[d] & S(\gamma^4) \arrow[d] \\
S^1 \arrow[r] & S^5 \arrow[r,"\pi"] & \mathbb{C}P^2 \arrow[r,"v"] & \mathbb{C}P^2\vee S^4 \arrow[r,"\widehat{\gamma}\vee f_{k,l}"] & BSO(4)
\end{tikzcd}
\]
Here $\gamma^4$ is the universal $4$-dimensional vector bundle, $\widehat{\gamma}$ and $f_{k,l}$ are classifying maps of $\gamma\oplus\epsilon^2$ and $\xi_{k,l}$, respectively, and note that $f_{k,l}=ki_*b+ l\sigma_*\iota_3$, using our previous notations in Lemma \ref{thm:5}. Thus $E_{k,l}$ and $E'_{k,l}$ are actually total spaces of $S^3$-bundles over $S^5$. To determine them, we only need to determine the classifying maps.

We will identify $\pi_n(BSO(m))$ with $\pi_{n-1}(SO(m))$. The corresponding element of $\alpha\in \pi_{n-1}(SO(m))$ in $\pi_n(BSO(m))$ will be denoted by $\widetilde{\alpha}$.

\begin{prop}
  \begin{enumerate}
    \item $f_{k,l}q\pi=0$;
    \item $(\widehat{\gamma}\vee f_{k,l})v\pi=k\widetilde{i_*b\eta_3}+l\widetilde{\sigma_*\eta_3}$.
  \end{enumerate}
\end{prop}

\begin{proof}
  \begin{enumerate}
    \item $\mathbb{C}P^3/\mathbb{C}P^1\simeq S^4\cup_{q\pi}S^6$. As $Sq^2=0: H^4(\mathbb{C}P^3;\mathbb{Z}_2)\to H^6(\mathbb{C}P^3;\mathbb{Z}_2)$, it follows that $q\pi=0$.
    \item $\pi_5(\mathbb{C}P^2\vee S^4)\cong \mathbb{Z}\{\pi\}\oplus\mathbb{Z}_2\{\eta_4\}\oplus\mathbb{Z}\{[\iota_2,\iota_4]\}$. It is easy to see that $v\pi=\pi\pm [\iota_2,\iota_4]$. Then
        \[
        (\widehat{\gamma}\vee f_{k,l})v\pi=(\widehat{\gamma}\vee f_{k,l})(\pi\pm [\iota_2,\iota_4])=\widehat{\gamma}\pi\pm [\widetilde{\epsilon_1},k\widetilde{i_*b}+l\widetilde{\sigma_*\iota_3}],
        \]
        where $\pi_1(SO(4))\cong \mathbb{Z}_2\{\epsilon_1\}$. $\widehat{\gamma}\pi=0$, as $\gamma$ can be extended to a bundle over $\mathbb{C}P^3$, i.e. the tautological line bundle over $\mathbb{C}P^3$. Thus
        \[
        (\widehat{\gamma}\vee f_{k,l})v\pi=k[\widetilde{\epsilon_1},\widetilde{i_*b}]+ l[\widetilde{\epsilon_1},\widetilde{\sigma_*\iota_3}]=k\widetilde{\langle \epsilon_1,i_*b\rangle}+l\widetilde{\langle \epsilon_1, \sigma_*\iota_3 \rangle },
        \]
        where $\langle\cdot , \cdot\rangle$ denote the Samelson pruduct. Similar arguments as in the proof of \cite [theorem 1.1]{KKKT07} shows that $\langle \epsilon_1, \sigma_*\iota_3 \rangle=\sigma_*\eta_3$ and $\langle \epsilon_1,i_*b\rangle=i_*b\eta_3$. Hence the result follows.
  \end{enumerate}
\end{proof}

Immediately, we have

\begin{cor}
  \begin{enumerate}
    \item $E_{k,l}\cong S^3\times S^5$;
    \item $E'_{k,l}\cong
    \begin{cases}
      S^3\times S^5, & \mbox{if $k,l$ are even,}\\
      M_{1,0}, & \mbox{if $k$ is odd, $l$ is even,} \\
      SU(3), & \mbox{if $l$ is odd}.
    \end{cases}$
  \end{enumerate}
  Consequently, all topological $8$-manifolds of type ($\ast$) admit infinite free $S^1$-actions.
\end{cor}

\begin{cor}
  \begin{enumerate}
    \item $\pi_4(N_{k,l})\cong \mathbb{Z}_2$;
    \item $\pi_4(N'_{k,l})\cong
    \begin{cases}
      0, & \mbox{if $l$ is odd,} \\
      \mathbb{Z}_2, & \mbox{if $l$ is even}.
    \end{cases}$
  \end{enumerate}
\end{cor}

\begin{rmk}
  We do not know whether $S^3\times S^5\#\Sigma^8$ admits a smooth free $S^1$-action.
\end{rmk}

The second class of manifolds are certain $S^1$-bundles over $S^2$-bundles over $\mathbb{C}P^2$. $S^1$-bundles over these manifolds are just $T^2$-bundles over $S^2$-bundles over $\mathbb{C}P^2$. $T^2$-bundles are determined by two elements in $H^2$. Different pairs of elements which are in the same orbit under $SL(2,\mathbb{Z})$-actions give isomorphic $T^2$-bundles. As the second cohomology groups of $S^2$-bundles over $\mathbb{C}P^2$ are all isomorphic to $\mathbb{Z}\oplus\mathbb{Z}$, there is a unique $T^2$-bundle over each $S^2$-bundle over $\mathbb{C}P^2$ with $1$-connected total space, up to isomorphism. Thus $S^1$-bundles over such $7$-manifolds are only depend on the based $S^2$-bundles over $\mathbb{C}P^2$.

A $3$-dimensional vector bundle over $\mathbb{C}P^2$ is determined by its second Stiefel-Whitney class and first Pontryagin class, and can be constructed as follows:
\[
\begin{matrix}
  \text{Spin case:} & \quad & \text{Nonspin case:} \\
  \begin{tikzcd}
q^*(\xi_{k}) \arrow[r] \arrow[d] & \xi_{k} \arrow[d] \\
\mathbb{C}P^2 \arrow[r,"q"] & S^4
\end{tikzcd} & \quad &
\begin{tikzcd}
v^*(\xi'_{k}) \arrow[r] \arrow[d] & \xi'_{k} \arrow[d] \\
\mathbb{C}P^2 \arrow[r,"v"] & \mathbb{C}P^2 \vee S^4
\end{tikzcd}
\end{matrix}
\]
where $\xi_k\oplus\epsilon\cong \xi_{k,0}$ and $\xi'_k\oplus\epsilon\cong \xi'_{k,0}$. Let $L_k=S(\xi_k)$, $L'_k=S(\xi'_k)$. Gysin sequences show that $H^2(L_k)\cong\mathbb{Z}\{x_k\}\oplus\mathbb{Z}\{y_k\}$, $H^2(L'_k)\cong\mathbb{Z}\{x'_k\}\oplus\mathbb{Z}\{y'_k\}$, with $x_k$ and $x'_k$ pulled back from a generator of $H^2(\mathbb{C}P^2)$.

Now consider the $S^1$-bundles over $L_k$ and $L'_k$ with Euler class $x_k$ and $x'_k$, and write their total spaces as $P_k$ and $P'_k$, respectively. Then we have the following pullback squares of fiber bundles:
\[
\begin{tikzcd}
 & S^2 \arrow[d] & S^2 \arrow[d] & S^2 \arrow[d] & S^2 \arrow[d] \\
S^1 \arrow[r] & P_k \arrow[r] \arrow[d] & L_k \arrow[d] \arrow[r] & S(\xi_{k}) \arrow[r] \arrow[d] & S(\gamma^3) \arrow[d] \\
S^1 \arrow[r] & S^5 \arrow[r,"\pi"] & \mathbb{C}P^2 \arrow[r,"q"] & S^4 \arrow[r,"f_{k}"] & BSO(3)
\end{tikzcd}
\]
\[
\begin{tikzcd}
 & S^2 \arrow[d] & S^2 \arrow[d] & S^2 \arrow[d] & S^2 \arrow[d] \\
S^1 \arrow[r] & P'_k \arrow[r] \arrow[d] & L'_k \arrow[d] \arrow[r] & S(\xi'_{k}) \arrow[r] \arrow[d] & S(\gamma^3) \arrow[d] \\
S^1 \arrow[r] & S^5 \arrow[r,"\pi"] & \mathbb{C}P^2 \arrow[r,"v"] & \mathbb{C}P^2\vee S^4 \arrow[r,"\widehat{\gamma}\vee f_{k}"] & BSO(3)
\end{tikzcd}
\]
Similar arguments as the first example shows that
\[
f_kq\pi=0, (\widehat{\gamma}\vee f_k)v\pi=k\widetilde{b\eta_3}.
\]
Thus
\[
P_k\cong S^5\times S^2, P'_k\cong\begin{cases}
                  S^5\times S^2, & \mbox{if } k \text{ is even} \\
                  S^5\widetilde{\times} S^2, & \mbox{otherwise}.
                \end{cases}
\]
Since $S^5\widetilde{\times} S^2\simeq S^2\cup_{\eta_2\eta_3}e^5\cup e^7$, $\pi_4(S^5\widetilde{\times} S^2)\cong \pi_4(S^2\cup_{\eta_2\eta_3}e^5)=0$. It follows that the total space of the $S^1$-bundle over $S^5\widetilde{\times} S^2$ with Euler class a generator of $H^2(S^5\widetilde{\times} S^2)$ is diffeomorphic to $SU(3)$. For $S^5\times S^2$, it is obviously $S^5\times S^3$. Hence we have

\begin{cor}
  There are infinite free smooth $T^2$-actions on $S^3\times S^5$ and $SU(3)$.
\end{cor}

\begin{rmk}
  Finally, we list some mistakes in \cite{EZ14} related to our discussion:
  \begin{enumerate}
    \item In p. 27, they stated that the only homotopy types of an $8$-manifold of type ($\ast$) which admits free $S^1$-actions are $S^3\times S^5$ and $SU(3)$. It turns out that $M_{1,0}$ also admit free $S^1$-actions and not homotopy equivalent to any of them.
    \item The remark in p. 34 and the second remark in p. 38 is true. Be careful that their notations are different from ours.
    \item The remark in p. 45 is incorrect. $\pi_4$ will always be $\mathbb{Z}_2$.
  \end{enumerate}
\end{rmk}

\bibliography{bib.bib}

\end{document}